\documentclass[prepint]{imsart}
\usepackage{enumerate}
\usepackage{appendix}
      \usepackage{amssymb}
\usepackage{color}
\usepackage{dsfont}
\usepackage[pdftex]{graphicx}
\usepackage{amsmath}
\usepackage{mathtools}
\RequirePackage[OT1]{fontenc}
\RequirePackage{amsthm,amsmath}
\RequirePackage[numbers]{natbib}
\RequirePackage[colorlinks,citecolor=blue,urlcolor=blue]{hyperref}

\startlocaldefs
      \numberwithin{equation}{section}
\theoremstyle{plain}
\newtheorem{theorem}{Theorem}[section]
      
      \newtheorem{proposition}[theorem]{Proposition}
      \newtheorem{corollary}[theorem]{Corollary}
      \newtheorem{conjecture}[theorem]{Conjecture}
      \newtheorem{definition}[theorem]{Definition}
      \newtheorem{remark}[theorem]{Remark}
\endlocaldefs

\begin{document}

\begin{frontmatter}

% "Title of the paper"
\title{On the connection between transient and ballistic behaviours for RWRE}
\runtitle{on the RWRE conjecture}

% indicate corresponding author with \corref{}
% \author{\fnms{John} \snm{Smith}\corref{}\ead[label=e1]{smith@foo.com}\thanksref{t1}}
% \thankstext{t1}{Thanks to somebody}
% \address{line 1\\ line 2\\ printead{e1}}
% \affiliation{Some University}

\author{\fnms{Enrique} \snm{Guerra Aguilar}\ead[label=e1]{eaguerra@mat.puc.cl}\thanksref{t1} }
\thankstext{t1}{Supported by CONICYT FONDECYT Postdoctorado 3180255}
\affiliation{Pontificia Universidad Cat\'{o}lica de Chile}

\runauthor{E. Guerra}

\begin{abstract}
We study the strong form of the ballistic conjecture for random walks in random environments (RWRE). This conjecture asserts that any RWRE which is directionally transient for a nonempty open set of directions satisfies condition $(T)$ (annealed exponential decay for the unlikely exit probability).
Specifically, we introduce a ballisticity condition which is fulfilled as soon as a polynomial condition of degree greater than $d-1$ holds. Under that hypothesis we prove condition $(T)$, which turns this condition into the weakest-known ballisticity assumption. We recall that standard arguments to prove that a ballisticity condition implies directional transience require at least polynomial decay greater than degree $d$. Furthermore, in the one dimensional case we provide an alternative proof which proves the equivalence between transient behaviour and annealed arbitrary decay for the unlikely exit probability, we expect that this new argument might be used in higher dimensions.

\end{abstract}



\end{frontmatter}

\section{Introduction}
In the higher dimensional case $d\geq2$, it is conjectured that any $d-$ dimensional random walk in an i.i.d. uniform elliptic random environment (RWRE) which is transient along an open set of directions, is also ballistic. Alongside, the so-called condition $(T)$ introduced by Sznitman in \cite{Sz01} has shown to be an important assumption in order to quantify ballistic regime. For instance, assuming condition $(T)$ one finds that ballistic behaviour, functional central limit theorem and large deviation estimates are fulfilled (cf. \cite{Sz01} and \cite{Sz02}). Moreover, condition $(T)$ is equivalent to transient behaviour in the one dimensional case.

\noindent
Somehow, condition $(T)$ tries to quantify the gap needed in order to prove the conjecture. Indeed, the strong form of the previous conjecture is expected: "transience along an open set of directions implies condition (T)". Our main objective is trying to interpolate between these two behaviours: directional transience along an open set and condition $(T)$.
We introduce the weakest-known ballisticity condition and we prove that under that condition, the stronger condition (T) is fulfilled. Our weak ballisticity condition will be satisfied under a polynomial condition of degree $d-1$.

\noindent
We recall that previous related results can be found in \cite{BDR14} where the authors proved a similar result for a polynomial decay of degree at least $15d+5$, and in \cite{GVV19} for a degree of $9d$. To the best of my knowledge, the most standard proof to show that a ballisticity condition implies transience requires a polynomial decay of degree at least $d$ (cf. Lemma 3.38 in \cite{DR14}). On the other hand, our proof makes an exhaustive use of the previous techniques developed in \cite{BDR14} and \cite{GVV19}, and it is likely to expect that the underlying procedure cannot be reproduced to improve on the starting decay. Trying to solve in part this problem, we provide an alternative proof in the one dimensional case which might be extended to higher dimensional setting.

\smallskip

Let us introduce the standard setting in order to properly explain the previous informal discussion. We let the underlying dimension $d\geq1$, and notice that the environment prescribes at each site in $\mathbb Z^d$ the transitions governing the evolution of the random particle. Specifically we let $\kappa\in (0,1/(2d)]$ and define the simplex:
\begin{equation}
\label{simplex}
\mathcal P_\kappa:=\left\{z\in \mathbb R^{2d}: \sum_{i=1}^{2d}z_i=1, \ z_i\geq \kappa \ \forall i\in[1,2d] \right\}.
\end{equation}
We will denote norms $\ell^1$ and $\ell^2$, by $|\ \cdot \ |_1$ and $|\ \cdot \ |_2$, respectively.

\noindent
The set of environments is $\Omega:=\mathcal P_\kappa^{\mathbb Z^{d}}$ and we denote an element $\omega\in \Omega$ in the form $\omega:=\omega(x,e)=\omega(x,\cdot)$, $x\in\mathbb Z^d$, $e\in\mathbb Z^d$,  with $|e|_1=1$. We also use the notation $\omega_x:=\omega(x,\cdot)$, for $x\in\mathbb Z^d$.

\smallskip
For the time being, assume a given ergodic probability measure $\mathbb P$ on the $\sigma-$ algebra $\mathfrak F_{\Omega}$, generated by cylinder sets in $\Omega$. Let $\omega\in\Omega$ and $x\in \mathbb Z^d$ and define the \textit{quenched law} $P_{x,\omega}$ as the probability measure of the Markov chain $(X_n)_{n\geq 0}$ with state space in $\mathbb Z^d$ starting from $x$ and stationary transition probabilities to nearest neighbour sites, given by the environment, i.e.
\begin{align*}
&P_{x,\omega}[X_0=x]=1 \mbox{  and}\\
&P_{x,\omega}\left[X_{n+1}=X_n+e|X_n\right]=\omega(X_n,e),\,\,\mbox{ for }e \in \mathbb Z^d \,\,\mbox{ with }\,\, |e|=1.
\end{align*}

\noindent
We then define for $x\in \mathbb Z^d$ the \textit{annealed} probability measure $P_x$ via the semidirect product $P_x:=\mathbb P\times P_{x,\omega}$ on $\Omega\times (\mathbb Z^d)^{\mathbb N}$ endowed with its canonical $\sigma-$ algebra. With a little abuse of notation, we will denote as well by $P_x$, the marginal law of the process $(X_n)_{n\geq0}$ under $P_x$ itself.  We use symbols $(\mathcal F_n)_{n\geq0}$ and $\mathcal F$ to indicate the natural filtration and $\sigma-$ algebra of the random walk process, respectively.

\smallskip
We study the RWRE in strong mixing random environments, following certain extension of X. Guo in \cite{Gu14}. For a universal set $U$, and a subset $A\subset U$ we write $U\setminus A$ the complement of $A$, and we simply write this by $A^c$ whenever $U$ is clear from the context.

We use the notation $|\,\cdot\,|_1$ and $|\,\cdot\,|_2$ to denote the $\ell_1$ and $\ell_2$-distance on $\mathbb R^d$ respectively; and furthermore, for $A, B\subset \mathbb Z^d$, $i\in \{1,2\}$, the notation $d_i(A,B)$ stands for the canonical $\ell_i$-distance between sets $A,\, B$, i.e. $d_i(A,B):=\inf\{|x-y|_i,\, x\in A, y\in B\}$.

We first define a Markovian field on the lattice $\mathbb Z^d$.
\begin{definition}{Markovian field on $\mathbb Z^d$}\label{defmark}
For $r\geq 1$ and $V \subset \mathbb{Z}^d$, let $\partial^r V=\{z\in V^c: d_1(z, V)\leq r\}$ be the $r-$ boundary of the set $V$. To simplify notation we will also write $\partial^1 V = \partial V$ for sets $V \subset \mathbb{Z}^d$. A random environment $(\mathbb P, \mathfrak{F}_{\Omega})$ on $\mathbb Z^d$ is called $r$-Markovian if for any finite $V\subset \mathbb Z^d$, $\mathbb P-$ a.s.
\begin{equation*}
  \mathbb P[(\omega_{x})_{x\in V}\in \cdot|\mathfrak{F}_{V^c}]=\mathbb P[(\omega_x)_{x\in V}\in \cdot|\mathfrak{F}_{\partial ^r V}],
\end{equation*}
where $\mathfrak{F}_{\Lambda}=\sigma(\omega_x, \, x\in \Lambda)$.
\end{definition}
We then introduce the strong mixing assumption.
\begin{definition}{Strong mixing environments}\label{defsma}
 Let $C$ and $g$ be positive real numbers. We will say that an $r$ -Markovian field $(\mathbb P, \mathfrak{F}_{\Omega} )$ satisfies the strong mixing condition \textbf{(SM)}$_{C,g}$ if for all finite subsets $\Delta\subset V \subset \mathbb Z^d$ with $d_1(\Delta, V^c)\geq r$, and $A\subset V^c$,
\begin{equation}
\label{sma}
\frac{d\mathbb P[(\omega_x )_{x\in \Delta}\in \cdot | \eta]}{d \mathbb P[(\omega_x )_{x\in \Delta}\in \cdot | \eta']}\leq \exp\left( C \sum_{x\in \partial^r \Delta, y \in \partial^r A}e^{-g|x-y|_1}\right)
\end{equation}
for $\mathbb P-$ a.s. all pairs of configurations $\eta, \,\eta'\in \mathcal P_{\kappa}^{\mathbb Z^d} $ which agree over the set $V^c \backslash A$. Here we have used the notation
\begin{equation*}
\mathbb P[(\omega_x )_{x\in \Delta}\in \cdot | \eta]=\mathbb P[(\omega_x )_{x\in \Delta}\in \cdot |\mathfrak{F}_{V^c}]|_{(\omega_x)_{x\in V^c}=\eta}.
\end{equation*}

\end{definition}

\smallskip

\vspace{0.5ex}
We introduce the so-called ballisticity conditions, nevertheless we first need to establish some further terminology. We define the unit sphere $\mathbb S^{d-1}$ by
$$
\mathbb S^{d-1}:=\{x\in \mathbb R^d:\ \sum_{i=1}^{d} x_i^2=1\}.
$$
We then define for $L\in \mathbb R$ and $\ell\in \mathcal S^{d-1}$, the following $\left(\mathcal F_n\right)_{n\geq0}$ -stopping times:
\begin{gather}
\nonumber
T_L^\ell:=\inf\{n\geq 0:\,X_n\cdot \ell\geq L \}\hspace{1.5ex}\mbox{and}\\
\label{stopleftright}
\widetilde T_{L}^\ell:=\inf\{n\geq0 \, X_n\cdot \ell \leq L\}.
\end{gather}
We define classic Sznitman $T-$ types of ballisticity conditions.
\begin{definition}\label{defTs}
Let $\gamma\in(0,1]$ and $\ell\in \mathbb S^{d-1}$. We say that condition $(T^\gamma)|\ell$ holds, if for each $b>0$ there exists some neighbourhood $\mathcal U_\ell$ of $\ell$ in $\mathbb S^{d-1}$, such that for each $\ell' \in \mathcal U_\ell$,
\begin{equation}\label{defTgammaeq}
\limsup_{\substack{L\rightarrow\infty}}L^{-\gamma}\log P_0\left[\widetilde T_{-bL}^{\ell'}<T_L^{\ell'}\right]<0
\end{equation}
is fulfilled. We further define condition $(T)|\ell$ as simply $(T^1)|\ell$, and condition $(T')|\ell$ as the requirement that $(T^\gamma)|\ell$ is fulfilled for each $\gamma\in (0,1)$.
\end{definition}
We introduce a priori weaker polynomial ballisticity conditions and transient behavior as follows
\begin{definition}\label{defPOL}
Let $\ell\in\mathbb S^{d-1}$ and $M>0$, we say that the RWRE satisfies condition $(\mathcal P^M)|\ell$ if for each $b>0$ there exists some neighbourhood $\mathcal U_\ell$ of $\ell$ in $\mathbb S^{d-1}$, such that for each $\ell' \in \mathcal U_\ell$,
\begin{equation}\label{defPoleq}
\lim_{\substack{L\rightarrow\infty}} L^{M}P_0\left[\widetilde T_{-bL}^{\ell'}<T_L^{\ell'}\right]=0.
\end{equation}
Furthermore, we say that the RWRE is transient along $\ell$, whenever
\begin{equation}\label{deftraneq}
P\left[\lim_{n\rightarrow\infty}X_n\cdot\ell=\infty\right]=1.
\end{equation}
\end{definition}
We now introduce the definition of ballistic asymptotic behaviour:
\begin{definition}{Non-vanishing limiting velocity}\label{defbullet}
Let $\ell\in \mathbb S^{d-1}$. We say that the RWRE satisfies a ballistic strong law of large numbers along direction $\ell$, if there exists a deterministic non-vanishing velocity $v\in \mathbb R^d$ with $v\cdot \ell>0$ such that $P_0-$ a.s.
\begin{equation}\label{velocity}
\lim_{\substack{n\rightarrow\infty}}\frac{X_n}{n}=v.
\end{equation}
\end{definition}
 We consider the direct product case $\mathbb P=\mu^{\otimes\mathbb Z^d}$ for certain fixed probability $\mu$ on the canonical $\sigma-$ algebra for set $\mathcal P_{\kappa}$. We call this environmental framework an i.i.d. random environment. Then the fundamental conjecture can be settled as the following assertion.
\begin{conjecture}[$d\geq2$]\label{conjec1}
Let $\ell\in \mathbb S^{d-1}$, then for any random walk in an i.i.d. uniform elliptic random environment the following assertions are equivalents:
\begin{enumerate}[(i)]
\item Directional transience along each direction in a nonempty open set $\mathcal U_{\ell}\subset \mathbb S^{d-1}$ is fulfilled, with $\ell\in\mathcal U_{\ell}$.
\item A ballistic strong law of large numbers along direction $\ell$ with velocity $v\in \mathbb R^d$ holds.
\end{enumerate}
\end{conjecture}
Notice that by Theorem 3.6 in \cite{Sz01}, the following conjecture is indeed a stronger form.
\begin{conjecture}\label{conjec2}
Let $\ell\in \mathbb S^{d-1}$, then for any random walk in a strong mixing uniform elliptic random environment the following assertions are equivalents:
\begin{enumerate}[(i)]
\item Directional transience along an open set $\mathcal U_{\ell}\subset \mathbb S^{d-1}$ is fulfilled, with $\ell\in\mathcal U_{\ell}$.
\item $(T)|\ell$ is fulfilled.
\end{enumerate}
\end{conjecture}
We remark that by the main result in \cite{GR18} and Theorem 1.8 in \cite{GVV19}, the condition $(T)|\ell$ above can be replaced by $(\mathcal P^M)|\ell$, with $M>9d$. Indeed, this work proves that a further weaker decay can be considered as equivalent to condition $(T)|\ell$. For a set $A\subset \mathbb Z^d$ we introduce $(\mathcal F_n)_{n\geq0}-$ stopping times
\begin{align}
\label{stopexitset}
T_A:=&\inf\{n\geq0:\ X_n\notin A\},\hspace{2ex}\mbox{along with}\\
\nonumber
H_A:=&\inf\{n\geq0:\ X_n\in A\},
\end{align}
which are the first exit and entrance time to set $A$.
\begin{definition}\label{defweakcond}
Let $c>0$, $L>0$, $\ell\in\mathbb S^{d-1}$ and $R$ be a rotation of $\mathbb R^d$ such that with $R(e_1)=\ell$. Define blocks
\begin{align}
&\widetilde B_{1}(c,L)=R\left([0,L]\times[0,3cL]^{d-1}\right)\cap\mathbb Z^d,\\
&B_{2}(c,L)=R\left((-L,(1+(1/11))L)\times (-cL,4cL)\right)\cap\mathbb Z^d.
\end{align}
We also define the frontal part of the $B_2(c,L)-$ boundary, via
$$
\partial^+B_{2,L}:=\partial B_2(c,L)\cap\{z: z\cdot \ell\geq (1+(1/11))L\}.
$$
We say that condition $(\mathcal W_{c,M})|\ell$ is satisfied if there exist some $c>0$, $M>1/\lambda_1$, such that
\begin{equation}\label{defweakcondeq}
\mathbb E\left[\sup_{x\in\widetilde B_{1}(c,M)}P_{x,\omega}\left[X_{T_{B_{2}(c,M)}}\notin \partial ^+B_2(c,M)\right]\right]<\lambda_1
\end{equation}
holds, where $\lambda_1<1$ is an absolute positive constant depending only on $d,\ \kappa,\ g,\ C$ and $r$.
\end{definition}
It is not hard to prove that this condition is implied by $(\mathcal P^{d-1})|\ell$.
The main theorem of this article is the following theorem.
\begin{theorem}\label{mainth1}
Assume condition $(\mathcal W_{c,M})|\ell$ for some constants $c>0$ and $M>1/\lambda_1$, then condition $(T')|\ell$ holds. Furthermore, if in addition the random environment has an i.i.d. structure condition $(T)|\ell$ is satisfied. Furthermore, condition $(\mathcal W_{c,M})|\ell$ is implied by $(\mathcal P^{d-1})|\ell$.
\end{theorem}
Arbitrary decay on $M$ of the probability involved in (\ref{defweakcondeq}) is commonly accepted be not enough so as to prove directly condition $(T)$. Nevertheless, in the one-dimensional is true as the following corollary will prove.
\begin{corollary}\label{coronedim}
In the one dimensional i.i.d. case, the following assertions are equivalents for any RWRE:
\begin{itemize}
\item There exist $L_0>0$ and a function $\varphi:[0,\infty]\rightarrow[0,\infty]$ with 
$$\lim_{M\rightarrow\infty}\varphi(M)=0,$$ 
such that for all $L\geq L_0$
 $$
 \mathbb P[X_{T_{U_{L}}}\notin \partial^+U_L]\leq \varphi(L),
 $$
 where $U_L:=\{x\in \mathbb Z:\ |x|_1 <L \}$ and $\partial^+U_L:=\{x\in\mathbb Z:\ |x|_1=L\}$.
\item $(\mathcal W)_{c,M})|e_1$ holds for some positive constants $c$ and $M$ (=:arbitrary decay for the unlikely exit probability from slabs).
\item Transient along direction $e_1$ holds
\item $(T)|e_1$ holds.
\end{itemize}
\end{corollary}
This corollary follows from the proof of Theorem \ref{mainth1}, however we give an alternative argument. We think it might work in higher dimensional cases as well. On the other hand, we notice that in the one dimensional this can be derived by a one-dimensional version of the effective criterion of \cite{Sz02}. 

\smallskip
We shall now outline the structure of this article. In the next section we prove our main result Theorem \ref{mainth1}. Section \ref{seconedim} contains an alternative proof for Corollary \ref{coronedim} which lays out a possible viewpoint to answer the stronger form of the conjecture.

\medskip

\section{Renormalization scheme: Proof of Theorem \ref{mainth1}}
\label{secren}

We mainly aim in this section to construct a re-scaling method turning out stronger or sharper estimates starting from weaker ones. Commonly, these type of theoretical constructions are called renormalization procedures. In order to the entire process works, we need a so-called \textit{seed estimate}, along with an inductive estimate to pass from scale $k$ to $k+1$, for any integer $k\geq0$. The seed estimate will be condition $(\mathcal W_{c,M})|\ell$, for certain positive constants $c$, $M$ and $\ell\in\mathbb
S^{d-1}$. We will also obtain the meaning of the constant's model $\lambda_1$, even though we will not give its precise value.

\smallskip
\noindent
Throughout this section we fix a direction $\ell\in\mathbb S^{d-1}$ and a rotation $R$ of $\mathbb R^d$ such that $R(e_1)=\ell$.

We introduce the successive dimensions of the boxes involved in the corresponding scales.

\smallskip
Specifically, we consider sequences $(L_k)_{k\geq0}$ and $(\widetilde L_k)_{k\geq0}$:
\begin{align}
\label{scale1}
&3\sqrt d <L_0<L_1,\ \ N_0:=\frac{L_1}{L_0}=1100d^3\in\mathbb N,\\
\label{scale2}
&3\sqrt d<\widetilde L_0=L_0<\widetilde L_1,\ \  \widetilde N_0:=\frac{\widetilde L_1}{\widetilde L_0}=11d^3N_0^2\in\mathbb N,\\
\label{scale3}
&\mbox{and for $k\geq1$, we define:  }L_{k+1}=N_0L_k, \  \ \widetilde L_{k+1}=\widetilde N_0\widetilde L_k.
\end{align}
Notice that we have for $k\geq 1$,
\begin{equation*}
L_k=N_0^k L_0, \ \widetilde L_k=\widetilde N_0^k\widetilde L_0, \mbox{ and  }\  \widetilde L_k<L_k^3.
\end{equation*}
Further restrictions on the scaling sequences $(L_k)_{k\geq0}$ and $(\widetilde L_k)_{k\geq0}$ will be prescribed later on.

\smallskip
We denote $\mathfrak L_k$ for integer $k\geq 0$, the set:
$$
\mathfrak L_k:=L_k\mathbb Z \times 3c\widetilde L_k\mathbb Z^{d-1}.
$$
Moreover, for integers $k\geq0$ and $x\in\mathfrak L_k$, we consider boxes $\widetilde B_1(x,\widetilde c, L_k, \widetilde L_k)$, $B_2(x,\widetilde c, L_k, \widetilde L_k)$ and its boundary frontal part $\partial^+B_2(x,\widetilde c,L_k, \widetilde L_k)$ defined by:
\begin{align*}
&\widetilde B_1(x,\widetilde c, L_k, \widetilde L_k):= R(x+[0,L_k]\times [0,3\widetilde c  \widetilde L_k]^{d-1}))\cap\mathbb Z^d,\\
&B_2(x,\widetilde c, L_k, \widetilde L_k):=R(x+(-L_k,L_k(1+1/11)\times (-\widetilde c \widetilde L_k, 4\widetilde c \widetilde L_k)^{d-1})\cap\mathbb Z^d,\\
&\mbox{along with}\\
&\partial^+B_2(x,\widetilde c,L_k, \widetilde L_k):=\partial B_{2}(x,\widetilde c,L_k, \widetilde L_k)\cap\{z\in\mathbb Z^d:\ (z-x)\geq L_k(1+1/11)\}.
\end{align*}
We introduce a further block $\dot{B}_1(x,\widetilde, L_k)$,
\begin{equation}\label{dotbox1}
\dot B_1(x,\widetilde c, L_k,\widetilde L_k):=R\left(x+(0,L_k)\times (0,3\widetilde c\widetilde L_k)^{d-1}\right)\cap\mathbb Z^d.
\end{equation}
It will be useful to consider the set of boxes in scale $k\geq0$, denoted by $\mathfrak B_k$
$$
\mathfrak B_k:=\left\{B_2(x,\widetilde c, L_k,\widetilde L_k),\ x\in\mathfrak L_k\right\}.
$$
\begin{remark}\label{remarkcover}
Let $k\geq0$ be an integer and $\widetilde c >0$.

We note that by the choice of scales given (\ref{scale1})-(\ref{scale3}) and the boxes constructed above, we have the following property:

\smallskip
For $k\geq 1$ and $x\in\mathfrak L_k$, consider for fixed $B_2(x,\widetilde c, L_k, \widetilde L_k)$, the set:
\begin{align*}
\mathfrak B_{2,L_k,x}:=\{ & \dot B_1(y,\widetilde c, L_{k-1},\widetilde L_{k-1}), \   y\in\mathfrak L_{k-1}, \\
&\mbox{  such that  } \dot B_1(y,\widetilde c,L_{k-1},\widetilde L_{k-1})\subset B_2(x,\widetilde c, L_k,\widetilde L_{k})\}.
\end{align*}
One can see that,
\begin{equation}\label{quasicov}
B_{2}(x,\widetilde c, L_k,\widetilde L_k)\subset \bigcup_{\substack{y\in \mathfrak L_{k-1}\\ \dot B_1(y,\widetilde c, L_{k-1},\widetilde L_{k-1})\in \mathfrak B_{2,L_k,x}}}\widetilde B_1(y,\widetilde c, L_{k-1},\widetilde L_{k-1}).
\end{equation}
The property prescribed in (\ref{quasicov}) will be called "quasi-cover property".
\end{remark}

\smallskip
Throughout this section, we will assume condition $(\mathcal W_{c,M})|\ell$ for certain $c,$ $M$ and direction $\ell$. We consider the sequences $(L_k)_{k\geq0}$ $(\widetilde L_k)_{k\geq0}$ satisfying (\ref{scale1})-(\ref{scale3}), where $M=L_0$. For easy in the writing and $k\geq0$ we define
\begin{align}
\label{boxesscalek}
\widetilde B_{1,k}(x)&:=\widetilde B_1(x,\widetilde c, L_k,\widetilde L_k),\ B_{2,k}(x):=B_2(x,\widetilde c,L_k,\widetilde L_k)\\
\nonumber
\dot B_{1,k}(x)&:=\dot B_1(x,\widetilde c, L_k,\widetilde L_k), \mbox{ and  }\partial^+ B_{2,k}(x):= \partial^+ B_{2}(x,\widetilde c, L_k,\widetilde L_k).\\
\end{align}
In the next definition we introduce the event \textit{Good box}. Notice that the value of the constant $\lambda_1$ will be clear along the Section proofs.
\begin{definition}[Good Box]\label{defgood}
For $x\in\mathfrak L_0$, we say that box $B_{2,0}(x)$ is $L_0-$ \textit{Good} if
$$
\sup_{\substack{x\in\widetilde B_{1,0}(x)}}P_{x,\omega}\left[X_{T_{B_{2,0}(x)}}\notin\partial^+B_{2,0}(x)\right]<\lambda_1^{\frac{1}{2}}.
$$
Otherwise, we say that the box $B_{2,0}(x)$ is $L_0-$ \textit{Bad}.

\smallskip
Recursively, for $k\geq1$ and $x\in \mathfrak L_k$, we say that box $B_{2,k}(x)$ is $L_k-$ \textit{Good} if:

\smallskip
\noindent
There exists a box $B_{2,k-1}(y)\in\mathfrak B_{k-1}$, $y\in\mathfrak L_{k-1}$, with $\dot{B}_{1,k-1}(y)\subset B_{2,k}(x)$, such that for any other box $B_{2,k-1}(z)\in\mathfrak B_{k-1}$, with $z\in\mathfrak L_{k-1}$, $\dot B_{1,k-1}(y)\subset B_{2,k}(x)$ and  $B_{2,k-1}(y)\cap B_{2,k-1}(z)=\varnothing$, we have that $B_{2,k-1}(z)$ is $L_{k-1}-$ \textit{Good}.
\noindent
Otherwise, we say that $B_{2,k}(x)$ is $L_k-$ \textit{Bad}.
\end{definition}
Roughly speaking, for $k\geq0$ and $x\in\mathfrak L_k$, the box $B_{2,k}(x)$ is $L_k-$ \textit{Good} whenever there is at most one box $B_{2,k-1}(y)$, $y\in\mathfrak L_{k-1}$ which is $L_{k-1}-$ \textit{Bad} and contained in $B_{2,k}(x)$.

\smallskip
The next remark will be useful in several parts of the remaining section.
\begin{remark}\label{remarkl1distane}
Notice that for integer $k\geq0$ and $x\in\mathfrak L_k$, the event "the box $B_{2,k}(x)$ is $L_k-$ \textit{Good}" depends at most on transitions in the set:
\begin{align}\label{transdepeB_2k}
&\mathcal B_{k,x}:=R\left(x+\left(-A_k, L_k+ \frac{A_k}{11}\right)\times \left(-\widetilde c \widetilde A_k,3\widetilde c \widetilde L_k+\widetilde c \widetilde A_k\right)^{d-1}\right)\cap\mathbb Z^d,\\
\nonumber
&\mbox{where   } A_k:=\sum_{i=0}^k L_i \mbox{  and  }\widetilde A_k:=\sum_{i=0}^k\widetilde L_i.
\end{align}
Moreover, we observe that for a box $B_{2,k}(x)$ as above, the number of boxes in $\mathfrak{B}_k$ intersecting it along a straight line along direction $\ell=R(e_1)$ is five: two at each direction $\pm\ell$ points out, besides itself. The remaining of the boxes $B_{2,k}(y)$, with centre $y\in\mathfrak L_k$ in the complementary set to the slab:
$$
\mathcal H_{x,k,1}:=\{z\in\mathbb R^d:\ |(z-x)\cdot \ell|\leq (5/2) L_k\},
$$
are at least separated $(10/11)L_k$ in $\ell^1-$ distance.

\smallskip
Analogously, for a straight line through direction $R(e_i)$, where $i\in[2,d]$ there exist at most three boxes in $\mathfrak B_k$ intersecting $B_{2,k}(x)$. The remaining boxes with centres in the complementary set to the slab:
$$
\mathcal H_{x,k,i}:=\{z\in\mathbb R^d:\ (-1/2)<(z-x)\cdot R(e_i)<(7/2)\widetilde L_k\}
$$
are at least separated $\widetilde L_k$ in terms of $\ell^1-$ distance.

\smallskip
We plainly have that for any integer $k\geq1$,
$$
A_{k-1}\leq (1/11)L_k,\ \widetilde A_{k-1}\leq (1/11)\widetilde L_{k}.
$$
As a result of the precedent discussions, for $k\geq1$ any disjoint boxes $B_{2,k-1}(y_1)$, $B_{2,k-1}(y_2)$ where the points $y_1,y_2\in\mathfrak L_{k-1}$ in the quasi-cover of $B_{2,k}$ (cf. Remark \ref{remarkcover}), its respective set of site transitions:
$$
\mathcal B_{k-1,y_1}\mbox{   and   }\mathcal B_{k-1,y_2},
$$
are at least separated $(9/11)L_k$ in $\ell^1-$ distance. This remark will be used to apply mixing condition (\ref{sma}), similarly as was mentioned in \cite{GVV19} Remark 3.
\end{remark}
Recall that we are assuming condition $(\mathcal W_{c,M})|\ell$ and tacitly we must find the value of $\lambda_1$ (cf. Definition \ref{defweakcond}).
\begin{proposition}\label{propannebad}
Let $k$ be a non-negative integer and $x\in \mathfrak L_k$. For $k=0$, and any $x\in\mathfrak L_k$, we have that
\begin{equation}\label{ine1badbox}
\mathbb P\left[B_{2,k}(x)\mbox{ is }L_k-\mbox{ \textit{Bad}}\right]\leq \lambda_1^{\frac{1}{2}}.
\end{equation}
Furthermore, for $k\geq1$ there exists a constant $\eta_1>0$ such that for any $x\in\mathfrak L_k$,
\begin{equation}\label{ine2badbox}
\mathbb P\left[B_{2,k}(x)\mbox{ is }L_k-\mbox{ \textit{Bad}}\right]\leq e^{-\eta_1 2^{k}}.
\end{equation}
\end{proposition}
\begin{proof}
Observe that (\ref{ine1badbox}) is a simple consequence of Chevyshev's inequality under assumption $(\mathcal W_{c,M})|\ell$. Hence, we turn to prove the inequality (\ref{ine2badbox}).

\noindent
For this end, it will be convenient to prove by induction that we have for any integer $k\geq0$ and $x\in\mathfrak L_k$, the inequality:
\begin{equation}\label{indbadboxanne}
\mathbb P\left[B_{2,k}(x)\mbox{ is }L_k-\mbox{ \textit{Bad}}\right]\leq e^{-c_k 2^{k}},
\end{equation}
where the sequence $(c_k)_{k\geq0}$ is defined as follows. We introduce the absolute constant (depending only on $d$)
\begin{equation}\label{lambda2}
\lambda_2:=\left(\frac{5}{3}\widetilde N_0\right)^{2(d-1)}\left(\frac{23}{11}N_0\right)^2
\end{equation}
and define the sequence (recall constants $C$, $g$ and $r$ in Definition \ref{defsma}):
\begin{align}\label{ckindanne}
c_0& :=\ln\left(1/\lambda_1^{\frac{1}{2}}\right), \mbox{  and for }k\geq0\\
\nonumber
c_{k+1}& :=c_k-\frac{\ln(\lambda_2)}{2^{k+1}}-\frac{\exp\left(-g(9/11)L_k\right)9r^{2d}L_k^{2}(6\widetilde c \widetilde L_k)^{2(d-1)}C}{2^{k+1}}.
\end{align}
Afterwards, we shall prove that there exists a constant $\nu_1>0$, such that
$$
\inf_{k\geq0}c_k>\nu_1,
$$
and this will end our proof. Notice that the case $k=0$ was already proven, thus we have to prove the inductive step. We assume that (\ref{indbadboxanne}) holds for $k\geq0$ and we will see that (\ref{indbadboxanne}) is satisfied when $k$ is replaced by $k+1$. We will assume that $x=0\in\mathbb Z^d$, the other cases can be analogously treated.

\smallskip
Observe now that using Definition \ref{defgood}, the event  "$B_{2,k+1}(0)$ is $L_{k+1}-$ \textit{Bad}" is contained in the following event:
\begin{align}\label{defmk}
\mathfrak M_k:=\left\{ \right.& \left. \exists B_{2,k}(y_1),\ B_{2,k}(y_2)\in\mathfrak B_k: \ \dot B_{1,k}(y_1),\ \ \dot B_{1,k}(y_2)\subset B_{2,k+1}(0),\right. \\
\nonumber
& \left. B_{2,k}(y_1)\cap B_{2,k}(y_2)=\varnothing, B_{2,k}(y_1),\ B_{2,k}(y_2) \mbox{  are  }L_k\mbox{  \textit{Bad}}\right\}.
\end{align}
We apply Remark \ref{remarkl1distane}, together with Definition \ref{defsma} to find that $\mathbb P[\mathfrak M_k]$ is bounded from above by
$$
\sum_{(y_1,y_2)\in \mathcal N_{2,k}} \Gamma_{M}(y_1,y_2)\mathbb P\left[B_{2,k}(y_1) \mbox{ is } L_{k}- \mbox{ \textit{Bad}}\right]\mathbb P\left[B_{2,k}(y_2) \mbox{ is } L_{k}- \mbox{ \textit{Bad}}\right],
$$
provided we define the set $\mathcal N_{2,k}$ as follows:
\begin{align*}
\mathcal N_{2,k}:=&\left\{(z_1,z_2)\in\mathfrak L_k\times \mathfrak L_k:\ \dot B_{1,k}(z_1),\ \dot B_{1,k}(z_2)\subset B_{2,k+1}(0),\right.\\
&\left.\ B_{2,k}(z_1)\cap B_{2,k}(z_2)=\varnothing \right\},
\end{align*}
along with, for $(y_1,y_2)\in\mathcal N_{2,k}$ we define the \textit{mixing correction} $ \Gamma_{M}(y_1,y_2)$ by (cf. Definitions \ref{defmark} and \ref{defsma} for notation),
\begin{align*}
\Gamma_{M}(y_1,y_2):=& \exp\left(\sum_{\substack{z_1\in \partial^r B_{2,k}(y_1)\\z_2\in \partial^r B_{2,k}(y_2)}}C e^{-g|y_1-y_2|_1}\right).
\end{align*}
Where we have assumed $L_0=M>10r$ (cf. Definition \ref{defmark}), in order to apply the mixing assumption of Definition \ref{defsma}.

\smallskip
We apply the induction hypothesis (\ref{indbadboxanne}) to get that
\begin{equation}\label{indineanne}
\mathbb P\left[B_{2,k}(y_1) \mbox{ is } L_{k}- \mbox{ \textit{Bad}}\right]\mathbb P\left[B_{2,k}(y_2) \mbox{ is } L_{k}- \mbox{ \textit{Bad}}\right]\leq e^{-c_k2^{k+1}},
\end{equation}
for each $(y_1,y_2)\in\mathcal N_{2,k}$. Using rough counting arguments we obtain,
\begin{align}\label{estanne}
\left|\mathcal N_{2,k}\right|&\leq \lambda_2 \mbox{cf. ((\ref{lambda2})-(\ref{scale1})-(\ref{scale2}))}\\
\nonumber
\Gamma_{M}&\leq\exp\left(\exp\left(-g(9/11)L_k\right)9r^{2d}L_k^{2}(6\widetilde c \widetilde L_k)^{2(d-1)}C\right),
\end{align}
where $|A|$ denotes the cardinality of set $A$. Observe that the last bound is uniform on $(y_1,y_2)\in\mathcal N_{2,k}$.

\smallskip
We combine (\ref{defmk}), the estimates in (\ref{estanne}) and the induction hypothesis (\ref{indineanne}) to get that $\mathbb P[B_{2,k+1}(0) \mbox{ is } L_{k+1}- \mbox{ \textit{Bad}}]$ is bounded from above by:
$$
\exp\left(-2^{k+1}\left(c_k-\frac{\ln(\lambda(d))}{2^{k+1}}-\frac{\exp\left(-g(9/11)L_k\right)9r^{2d}L_k^{2}(6\widetilde c \widetilde L_k)^{2(d-1)}C}{2^{k+1}}\right)\right).
$$
By the very definition of the constants $c_k,\ k\geq0$ in (\ref{ckindanne}), we have finished the proof of (\ref{indbadboxanne}). As was mentioned, it is convenient at this point to find $\nu_1>0$ such that:
\begin{equation}\label{ckposi}
\inf_{\substack{k\geq0}} c_k>\nu_1,
\end{equation}
whenever $L_0\geq \nu_1$. Nevertheless, note that whenever $L_0$ is chosen so that (recall $L_0=\widetilde L_0$, cf. (\ref{scale2})):
$$
\exp\left(-g(9/11)L_0\right)9L_0^{2}(6\widetilde c \widetilde L_0)^{2(d-1)}C<e^{-g(1/30)L_0}
$$
one has the following estimate for the series entering at the definition of sequence $(c_k)_{k\geq0}$ in (\ref{ckindanne}),
\begin{align*}
\inf_{\substack{k\geq 0}} c_k&\geq c_0-\left(\sum_{ k=1}^{\infty}\frac{\ln(\lambda_2)+e^{-g(1/30)L_0}}{2^k}\right)\\
&=\ln\left(\frac{1}{\phi^{\frac{1}{2}}(L_0)}\right)-\left(\ln(\lambda(d))+e^{-g(1/30)L_0}\right).
\end{align*}
Therefore we choose $1/\lambda_1^{\frac{1}{2}}=4\lambda_2$, and we plainly have there exists $\nu_1>0$, such that (\ref{ckposi}) holds whenever $L_0>\nu_1$. This ends the proof of all the required claims in the proposition.
\end{proof}
The next step into the renormalization construction will be to obtain a quenched estimate for the random walk exit from a given \textit{Good box}. This is the harder and more extensive part of our proof. As the proof shall depict, a more involved argument will be needed, when it is compared to the one given in \cite{GVV19}, Proposition and Section 5. Roughly speaking, in order to bound from above the unlikely exit by the boundary side where $-\ell$ points out, we avoid here the use of uniform ellipticity prescribed in (\ref{simplex}), instead we will successively apply the strong Markov property.
\begin{proposition}\label{propquengood}
Let $k$ be a non-negative integer and $x\in \mathfrak L_k$. Assume that the box $B_{2,k}(x)$ is $L_k-$ \textit{Good}, then there exists a constant $\eta_2>0$ such that
\begin{equation}\label{inepropquenesgood}
\sup_{\substack{y\in \widetilde B_{1,k}(x)}}P_{y,\omega}\left[X_{T_{B_{2,k}(x)}}\notin \partial^+ B_{2,k}(x)\right]\leq e^{-\eta_2 v^{k}},
\end{equation}
where $v_k:=\frac{N_0}{4}$
\end{proposition}
\begin{proof}
Let us prove by using induction the following claim:

\smallskip
\noindent
Let $(c_k)_{k\geq0}$ be a sequence defined by:
\begin{align}\label{defckquen}
c_k&:=\frac{1}{4^kL_0}\ln\left(\frac{1}{\lambda_1^{\frac{1}{2}}}\right),\ \ (k\geq0).
\end{align}
Then, for any $k\geq0$ and $x\in\mathfrak L_k$ we have that,
\begin{equation}\label{claimquengood}
\sup_{\substack{y\in \widetilde B_{1,k}(x)}}P_{y,\omega}\left[X_{T_{B_{2,k}(x)}}\notin \partial^+ B_{2,k}(x)\right]\leq e^{-c_k L_k}.
\end{equation}

\smallskip
We see that the assertion of Proposition \ref{propquengood} is implied by claim (\ref{claimquengood}), with constant $\eta_2:=L_0c_0$.

\noindent
We prove (\ref{claimquengood}) by induction on $k$. The case $k=0$ and $x\in\mathfrak L_0$ is straightforward using Definition \ref{defgood}. We indeed have the estimate,
$$
\sup_{\substack{y\in\widetilde B_{1,k}(x)}}P_{y,\omega}[X_{T_{B_{2,k}(x)}}\notin \partial^+B_{2,k}(x)]<e^{-\ln\left(\frac{1}{\phi^{\frac{1}{2}}(L_0)}\right)}=e^{-c_0v^0}.
$$
As a result, it suffices that we assume that (\ref{claimquengood}), and prove the analogous estimate (\ref{claimquengood}) when $k$ is switched by $k+1$.

\smallskip
We notice that by stationarity of the probability measure $\mathbb P$, the worst case to estimate (\ref{inepropquenesgood}) is $x=0$. Thus we can a do assume $x=0$ and also assume that the box $B_{2,k+1}(0)$ is $L_{k+1}-$ \textit{Good}. Consider the leftmost expression in (\ref{claimquengood}) when $k$ is replaced by $k+1$. We introduce the $(\mathcal F_n)_{n\geq0}-$ stopping times $\sigma_{u}^{+i}$ and $\sigma_u^{-i}$ for $u\in\mathbb R$ and integer $i\in[2,d]$
\begin{align}\label{stopsigmas}
\sigma_u^{+i}&:=\inf\{n\geq0:\ (X_n-X_0)\cdot R(e_i)\geq u\}, \mbox{  and}\\
\nonumber
\sigma_u^{-i}&:=\inf\{n\geq0:\ (X_n-X_0)\cdot R(e_i)\leq u\}.
\end{align}
It will be convenient to introduce the path space event $\mathcal I_k$ of \textit{lateral exit from the box} $B_{2,k+1}(0)$ (cf. (\ref{stopleftright}) and (\ref{stopexitset}) for notation)
$$
\mathcal I_k:=\left\{\exists i\in[2,d]:\ \sigma_{\widetilde c\widetilde L_{k+1}}^{+i}<T_{B_{2,k+1}(0)}, \mbox{  or  } \sigma_{-\widetilde c\widetilde L_{k+1}}^{-i}<T_{B_{2,k+1}(0)}\right\}.
$$
Observe that the following decomposition for any $y\in \widetilde B_{1,k+1}(0)$ is satisfied,
\begin{align}\label{decoquenest}
P_{y,\omega}[X_{T_{B_{2,k+1}(0)}}\notin\partial ^+ B_{2,k+1}(0)]\leq &P_{y,\omega}[\mathcal I_k]\\
\nonumber
&+P_{y,\omega}[\mathcal I_k^c\cap \{X_{T_{B_{2,k+1}(0)}}\cdot\ell\leq-L_{k+1}\}].
\end{align}
We will use the induction hypothesis to split the proof into getting suitable upper bounds for the expressions:
\begin{equation}\label{problat}
P_{y,\omega}[\mathcal I_k], \mbox{   and}
\end{equation}
\begin{equation}\label{probleft}
P_{y,\omega}[\mathcal I_k^c\cap \{X_{T_{B_{2,k+1}(0)}}\cdot\ell\leq -L_{k+1}\}].
\end{equation}
We begin with an estimate for the probability in (\ref{problat}). Notice first that for arbitrary $y\in\widetilde B_{1,k+1}(0)$, we can further decompose that probability as follows:
\begin{equation}\label{sumsigma}
P_{y,\omega}[\mathcal I_k]\leq \sum_{i=2}^{d}\left(P_{y,\omega}[\sigma_{\widetilde c\widetilde L_{k+1}}^{+i}<T_{B_{2,k}(0)}]+P_{y,\omega}[\sigma_{-\widetilde c\widetilde L_{k+1}}^{+i}<T_{B_{2,k}(0)}]\right).
\end{equation}
Following a close analysis as the argument to prove Proposition 5 in \cite{GVV19}, we will obtain an upper bound for the following probability
\begin{equation}\label{+2quench}
P_{y,\omega}[\sigma_{\widetilde c\widetilde L_{k+1}}^{+2}<T_{B_{2,k}(0)}].
\end{equation}
The other terms inside the sum in (\ref{sumsigma}) could be bounded by a similar argument. In order to bound the probability in (\ref{+2quench}), it will be useful to set
$$
n_k:=\frac{23}{11}N_0+1.
$$
Let us indicate that $n_k$ is the amount of successive boxes $B_{2,k}(z),\ z\in\mathfrak L_k$ along a straight line along direction $\ell$, such that $\dot B_{1,k}(z)\subset B_{2,k+1}(0)$. We introduce integer parameter $J_k$ %and $\overline L_k$,
\begin{align}
J_k:=&\left[\frac{\widetilde N_0}{4(n_k+1)}\right].%\mbox{  and}\\
%\overline L_k:=&4\widetilde c(n_k+1)\widetilde L_k.
\end{align}
%Since $J_k\overline L_k\leq \widetilde c \widetilde L_{k+1}$, for any $y\in\widetilde B_{1,k}(0)$ one sees that $P_{y,\omega}-$ a.s.
%\begin{equation}\label{jlklessexitb2}
%\{\sigma_{J_k\overline L_k}^{+2}<T_{B_{2,k+1}(0)}\}\subset \{\sigma_{\widetilde c \widetilde L_{k+1}}^{+2}<T_{B_{2,k+1}(0)}\}.
%\end{equation}
%For integer $j>0$, we introduce the set $c_\perp(j,k)$ defined by
%\begin{equation}\label{setcperpj}
%c_\perp(j,k):=\{z\in\mathbb Z^d:\ z\cdot R(e_2)\in\overline L_k[j,j+1),\ z\in B_{2,k+1}(0) \}.
%\end{equation}
%Let us denote by $|\ \cdot \ |_\perp$ the semi-norm on $\mathbb R^d$ given by
%$$
%|w|_\perp:=\sup_{i\in[2,d]}|R(e_i)|, \mbox{  for $w\in\mathbb R^d$.}
%$$
%It will be useful to introduce as well, for $j\in\mathbb Z$ the \textit{cylinder set},
%\begin{equation} \label{setcoverj}
%\overline c(j,k):=\{z\in\mathbb Z^d: \ \inf_{\substack{w\in c_\perp(j,k)}}|z-w|_\perp\leq \widetilde c n_k \widetilde L_k,\ z\in B_{2,k+1}(0) \}.
%\end{equation}
%Throughout the remaining of this part in the proof, and for easy of notation when $u\in\mathbb R$, %
We denote by $\sigma_u$, the stopping time $\sigma_{u}^{+2}$ and observe that a similar approach as in the first bound proven in Proposition 5.6 of \cite{GVV19} shows us that for an arbitrary point $y\in\widetilde B_{1,k+1}(0)$,
\begin{align}
\nonumber
&P_{y,\omega}[\sigma_{\widetilde c \widetilde L_{k+1}}<T_{B_{2,k+1}(0)}]\\
\nonumber
&\leq \prod_{i=0}^{\left[(J_k-2)/3\right]-3}\left(n_ke^{-c_k  L_k}\right)
\label{estlat}
&\leq \exp\left(-(J_k/8)(c_kL_k-\ln(n_k))\right).
\end{align}

\smallskip
The previous upper bound is also satisfied for other directions in the set $\{\pm R(e_i),i\in[2,d]\}$ in inequality (\ref{sumsigma}). Thus we find that for arbitrary $y\in\widetilde B_{1,k+1}(0)$,
\begin{equation}\label{inelatfin}
P_{y,\omega}[\mathcal I_k]\leq \exp\left(-(J_k/8)(c_kL_k-\ln(2(d-1)n_k))\right).
\end{equation}

\smallskip
We now turn to estimate the probability displayed in (\ref{probleft}). The main strategy will be the introduction of Markov chain techniques to avoid the use of uniform elliptic assumption (\ref{simplex}). The method will improve the analogous estimate in \cite{GVV19}, Proposition 5.6 of Section 5. By Definition \ref{defgood}, one can pick a box $B_{2,k}(y),\ y\in\mathfrak L_k$ composing the quasi-cover of $B_{2,k+1}(0)$ (cf. Remark \ref{remarkcover}), such that any other box composing the quasi-cover of box $B_{2,k+1}(0)$ and not intersecting box $B_{2,k}(y)$, is $L_k-$ \textit{Good}. Thus, let us start by introducing suitable terminology which localizes box $B_{2,k}(y)$.

\smallskip
For integer $i\in [1,N_0]$ and $k$ fixed as above, we define $\mathfrak B_{k,i}$ the set of boxes in $\mathfrak B_k$ at position $i$ towards direction $-\ell$ points out, as follows
$$
\mathfrak B_{k,i}:=\{B_{2,k}(w), w\in\mathfrak L_k, \ w\cdot \ell=-iL_k,\ \dot B_{1,k}(w)\subset B_{2,k+1}(0)\}.
$$
By hypothesis $B_{2,k+1}(0)$ is $L_{k+1}-$ \textit{Good}, thus Remark \ref{remarkl1distane} says that there exist at most five consecutive integers $i\in [1,N_0]$, such that the sets $\mathfrak B_{k,i}$ contain $L_k-$ \textit{Bad} boxes, and all another box composing a quasi-cover as in Remark \ref{remarkcover} is $L_k-$ \textit{Good}. Therefore, in the worst case of Definition \ref{defgood}, we can choose an index $\widetilde i\in[1,N_0]$ so that the sets $\mathfrak B_{k,i}$, with $i\in [\widetilde i, \widetilde i+4]$ contain all of five bad boxes along direction $\ell$.

\noindent
Note that there exists a further case, i.e. when the bad boxes along direction $\ell$ are located toward $+\ell$ points out, nevertheless our argument will show that in this case the estimates are sharper (cf. (\ref{leftquencase1}), comments below (\ref{inefincase1}) and Remark \ref{remarkonedim}).

\noindent
We split the argument into three cases:
\begin{enumerate}[(i)]
\item \label{case1quen} Case $\widetilde i\in [N_0-9, N_0]$.
\end{enumerate}
In this case, we will prove that for any $y\in\widetilde B_{1,k+1}(0)$,
\begin{align}\label{firstinecase1}
&P_{y,\omega}[\mathcal I_k^c,\ \{X_{T_{B_{2,k+1}(0)}}\cdot\ell\leq -L_{k+1}\}]\\
\nonumber
&\leq P_{y,\omega}[\mathcal I_k^c,\ \{X_{T_{B_{2,k+1}(0)}}\cdot\ell\leq -(N_0-9)L_{k}\}].
\end{align}
We need to introduce some further definitions in order to prove (\ref{firstinecase1}). Recall that we have a given box $B_{2,k+1}(0)$ which  is $L_{k+1}-$ \textit{Good}, an arbitrary point $y\in\widetilde B_{1,k+1}(0)$, we are assuming the induction hypothesis (\ref{claimquengood}) and thus the integer $k$ is fixed. For $i\in \mathbb Z$, we define the strip $\mathcal H_i$ by
$$
\mathcal H_i:=\{x\in\mathbb Z^d:\ \exists z\in \mathbb Z^d\ |x-z|_1=1,\ (z-iL_k)(x-iL_k)\leq0\}.
$$
Furthermore, we introduce the truncated strip $\widehat{ \mathcal H}_i$, defined by ($y\in\widetilde B_{1,k+1}(0)$ is fixed as above)
\begin{equation}\label{striptil}
\widehat{ \mathcal H}_i:=\left\{x\in\mathbb Z^d:\ \forall i\in[2,d] \ \ |(x-y)\cdot R(e_i)|<\widetilde c\widetilde L_{k+1}\right\}
\end{equation}
We also define a function $I:\mathbb Z^d\mapsto \mathbb Z$ such that $I(z)=i$ on  $\{x\in\mathbb Z^d:\ x\cdot\ell\in [iL_k-(L_k/2),iL_k+(L_k/2)) \}$. Notice that under our choice of $L_0$ in (\ref{scale1}), we have $I(z)=i$ for $z\in \mathcal H_i$. It will be useful as well to introduce a sequence $(V_n)_{n\geq0}$ of $(\mathcal F_n)_{n\geq0}-$ stopping times, recording the successive visits to different strips $\mathcal H_{i},\ i\in \mathcal Z$. We define recursively,
\begin{align*}
&V_0=0,\ V_1=\inf\left\{n\geq 0:\ X_n\in \mathcal H_{I(X_0)+1}\cup \mathcal H_{I(X_0)-1}\right\}, \mbox{  and for $j>1$}\\
&V_j=V_{j-1}+V_1\circ\theta_{V_{j-1}}.
\end{align*}
We define random variables $P_z$ and $Q_z$,
\begin{align*}
P_z(\omega):=P_{z,\omega}[X_{V_1}\in \mathcal H_{I(X_0)+1}]\mbox{  and  }Q_z(\omega):=P_{z,\omega}[X_{V_1} \in \mathcal H_{I(X_0)-1}].
\end{align*}
for $z\in\mathbb Z^d$ (notice that $P_z(\omega)+Q_z(\omega)=1$). For integer $i$ we further define the random variable $\rho_i$, via
\begin{equation}\label{defrhoi}
\rho_i(\omega):=\sup\left\{\frac{Q_z(\omega)}{P_z(\omega)},\ z\in\widetilde{\mathcal H}_i\right\}.
\end{equation}
For fixed $\omega \in \Omega$ and $w_0:=N_0(1+(1/11))$,let us now introduce a function $f_\omega:\mathbb Z\mapsto (0,\infty)$ such that
\begin{align}
\label{deffunquen}
&f_\omega(j)=0, \mbox{  for $j\geq w_0+1$}.\\
\nonumber
&f_\omega(j)=\sum_{\substack{j\leq n\leq w_0}}\prod_{\substack{n< m \leq w_0}} \rho_m^{-1}(\omega)  \mbox{  otherwise.}
\end{align}
Since the environment $\omega$ will remain fixed along the proof, with a little abuse of notation, we denote by $P_z,\ Q_z$ and $\rho_i$ the values of the same functions in $\omega$. We also drop $\omega$ from the environmental function $f_\omega$. In these terms, we claim that
\begin{equation}\label{claimleftquen}
P_{z_1,\omega}[\mathcal I_k^c,\ \{X_{T_{B_{2,k+1}(0)}}\cdot\ell\leq -(N_0-9)L_{k}\}]\leq \frac{f(0)}{f(-(N_0-9))}.
\end{equation}
for an arbitrary point $z_1\in \widetilde{\mathcal H}_0$ (recall that $\widetilde{\mathcal H}_0$ depends on $y\in\widetilde B_{1,k+1}(0)$, see (\ref{striptil})).

\smallskip
In order to prove claim (\ref{claimleftquen}), one follows a similar argument as in \cite{GVV19}, second part of the proof of Proposition 5.6 (see also \cite{Sz02}, Proposition 2.1 for the original argument).

\noindent
As a result, for any $z_1\in\widehat{\mathcal H}_0$
\begin{align}\label{leftquencase1}
&P_{z_1,\omega}[\mathcal I_k^c,\ \{X_{T_{B_{2,k+1}(0)}}\cdot\ell\leq -(N_0-9)L_{k}\}]\\
\nonumber
&\leq \frac{\sum_{\substack{0\leq n\leq w_0}}\prod_{\substack{n<j\leq w_0}}\rho_{j}^{-1}}{\prod_{\substack{-(N_0-9)<j\leq w_0}}\rho_j^{-1}}=\prod_{\substack{-(N_0-9)<j<0}}\sum_{\substack{0\leq n\leq w_0}}\prod_{\substack{0\leq j\leq n}}\rho_j.
\end{align}
Observe that for each point $z\in\widehat{ \mathcal H}_i,\ i\in [-N_0, N_0(1+(1/11)]$, there exists a point $u:=u(z)\in \widetilde B_{1,k}(v)$ for some $v\in\mathfrak L_k$ (a box composing the quasi-cover of box $B_{2,k+1}(0)$, cf. Remark \ref{remarkcover}), such that $|z-u|_1$ together with $u\cdot\ell\geq iN_k$. Therefore, in virtue of the precedent discussion and uniform ellipticity (\ref{simplex}), we have
\begin{equation}\label{inesuprho}
\rho_i\leq \sup_{\substack{x\in\widehat B_{i,k}}}\frac{\frac{1}{\kappa}P_{x,\omega}[X_{T_{B_{2,k}(v)}}\notin \partial ^+B_{2,k}(v)]}{1-\frac{1}{\kappa}P_{x,\omega}[X_{T_{B_{2,k}(v)}}\notin \partial ^+B_{2,k}(v)]},
\end{equation}
where for $i\in [-N_0,N_0(1+(1/11))]$, we have denoted by $\widehat B_{i,k}$ the set $\{\ x\in \widetilde B_{1,k}(v),\ \mbox{ some  }v\in\mathfrak L_k,\ \dot B_{1,k}(v)\subset B_{2,k+1}(0),\  v\cdot \ell=iL_k \}$. Combining the induction hypothesis (\ref{claimquengood}), (\ref{inesuprho}) and (\ref{leftquencase1}), we find that for arbitrary $z_1\in \widehat{\mathcal H}_0$
\begin{equation}\label{inefincase1}
P_{z_1,\omega}[\mathcal I_k^c,\ \{X_{T_{B_{2,k+1}(0)}}\cdot\ell\leq -(N_0-9)L_{k}\}]\leq \left(\frac{2}{\kappa}e^{-c_kL_k}\right)^{N_0-9},
\end{equation}
provided that $L_0\geq \nu_1$ for some constant $\nu_1>0$.

\noindent
It is now straightforward to see that the case of bad boxes located toward $+\ell$ points out is more handling.

\smallskip
We continue with defining $(\mathcal F_n)_{n\geq0}-$ stopping times (cf. (\ref{striptil}) for notation)
$$
T_0:=\inf\{n\geq0:\ X_n\in \widehat{\mathcal H}_0\},
$$
together with
\begin{equation}\label{stoptily}
\widetilde T_y=\inf\left\{n\geq0: \ |\left(X_n-y\right)\cdot R(e_j)|\geq\widetilde c \widetilde L_{k+1} \mbox{  for some  }j\in[2,d]  \right\}.
\end{equation}

Fix $y\in \widetilde B_{1,k+1}(0)$, we observe that on the set $\{\mathcal I_k^c,\ \{X_{T_{B_{2,k+1}(0)}}\cdot\ell\leq -(N_0-9)L_{k}\}$, $P_{y,\omega}-$ a.s. we have $T_0<T_{B_{2,k+1}(0)}$ (cf. (\ref{stopexitset}) for notation) and $T_0<\widetilde T_y$ (cf. (\ref{stoptily})), as a result of the strong Markov property and using inequality (\ref{firstinecase1}), for an arbitrary $y\in\widetilde B_{1,k+1}(0)$ we have that
\begin{align}
\label{inefinfincase1}
&P_{y,\omega}[\mathcal I_k^c,\ X_{T_{B_{2,k+1}(0)}}\cdot\ell\leq -L_{k+1}]\\
\nonumber
&\leq P_{y,\omega}[\mathcal I_k^c,\ X_{T_{B_{2,k+1}(0)}}\cdot\ell\leq -(N_0-9)L_{k}]\\
\nonumber
&\leq \sum_{z_1\in\widehat{\mathcal H}_0}P_{y,\omega}[T_0<T_{B_{2,k+1}(0)}\wedge \widetilde T_y,\ X_{T_0}=z_1]\\
\nonumber
&\times P_{z_1,\omega}[\mathcal I_k^c,\ \{X_{T_{B_{2,k+1}(0)}}\cdot\ell\leq -(N_0-9)L_{k}\}]\\
\nonumber
&\leq \sup_{\substack{z_1\in \widehat{ \mathcal H}_0}}P_{z_1,\omega}[\mathcal I_k^c,\ \{X_{T_{B_{2,k+1}(0)}}\cdot\ell\leq -(N_0-9)L_{k}\}]\stackrel{(\ref{inefincase1})}\leq \left(\frac{2}{\kappa}e^{-c_kL_k}\right)^{N_0-9}.
\end{align}

\begin{enumerate}[(i)]
\setcounter{enumi}{1}
\item \label{case2quen} Case $\widetilde i\in [1,4]$.
\end{enumerate}
In this case, we push the walk up to the last time it gets to truncated strip $\widehat{\mathcal H}_{-9}$ and then, we will perform a similar analysis as in case (\ref{case1quen}). We fix $y\in\widetilde B_{1,k+1}(0)$ and define for integer $u\in[-N_0, N_0(1+(1/11))]$, the random time
$$
\mathcal T_{u}:=\sup\{n\geq0:\ X_n\in \widehat{\mathcal H}_u\}.
$$
Notice that on the event $\{ \mathcal I_k^c,\ X_{T_{B_{2,k+1}(0)}}\cdot \ell\leq -L_{k+1}\}$, $P_{y,\omega}-$ a.s. we have $\mathcal T_{-9}<T_{B_{2,k+1}(0)}$ and
$\mathcal T_{-9}<\widetilde T_y$ (cf. (\ref{stoptily})). Thus, in particular on $\{ \mathcal I_k^c,\ X_{T_{B_{2,k+1}(0)}}\cdot \ell\leq -L_{k+1}\}$, the random time $\mathcal T_{-9}$ is $P_{y,\omega}-$ a.s. finite and moreover, using the Markov property we find that
\begin{align}
\label{inecase2first}
&P_{y,\omega}[\mathcal I_k^c,\ X_{T_{B_{2,k+1}(0)}}\cdot \ell\leq -L_{k+1}]\\
\nonumber
&=\sum_{\substack{n\geq0,\ z_1\in\widehat {\mathcal H}_{-9}}}P_{y,\omega}[\mathcal T_{-9}=n<T_{B_{2,k+1}(0)}\wedge \widetilde T_y,\ X_{\mathcal T_{-9}}=z_1 ]\\
\nonumber
&\times P_{z_1,\omega}[\mathcal I_k^c,\ X_{T_{B_{2,k+1}(0)}}\cdot \ell\leq -L_{k+1},\ \widetilde H_{\widehat{\mathcal H}_{-9}}=\infty]\\
\nonumber
&\leq \sup_{\substack{z_1\in \widehat{\mathcal H}_{-9}}}P_{z_1,\omega}[\mathcal I_k^c,\ X_{T_{B_{2,k+1}(0)}}\cdot \ell\leq -L_{k+1},\ \widetilde H_{\widehat{\mathcal H}_{-9}}=\infty],
\end{align}
provided that for a set $A\subset \mathbb Z^d$, we defined the stopping time $\widetilde H_A:=\inf\{n\geq1:\ X_n\in A\}$. Moreover, we observe that for any $z_1\in\widehat{ \mathcal H}_{-9}$, by the Markov property we have that

\begin{align}\label{inecase2}
&P_{z_1,\omega}[ \mathcal I_k^c,\ X_{T_{B_{2,k+1}(0)}}\cdot \ell \leq -L_{k+1},\  \widetilde H_{\widehat{\mathcal H}_{-9}}=\infty]\\
\nonumber
&\leq \sum_{\substack{z\in \widehat{\mathcal H}_{-11}}} E_{z_1,\omega}[\widetilde H_{\widehat{\mathcal H}_{-11}}<T_{B_{2,k+1}(0)}, X_{\widetilde H_{\widehat{\mathcal H}_{-11}}}= z ] \times P_{z,\omega}[\mathcal I_k^c, \widetilde H_{\widehat{\mathcal H}_{-N_0}}<\widetilde H_{\widehat{\mathcal H}_{-10}}] \\
\nonumber
&\leq \sup_{\substack{z_2\in \widehat{\mathcal H}_{-11}}}P_{z_2,\omega}[\mathcal I_k^c,\ \widetilde H_{\widehat{\mathcal H}_{-N_0}}<\widetilde H_{\widehat{\mathcal H}_{-10}}].
\end{align}
Using the last inequality of (\ref{inecase2}), we have for any $z_1\in\widetilde B_{1,k+1}(0)$,
\begin{align}\label{ineref}
&P_{z_1,\omega}[\mathcal I_k^c,\  X_{T_{B_{2,k+1}(0)}}\cdot\ell\leq -(N_0-9)L_{k}]\\
\nonumber
&\leq P_{z_1,\omega}[\mathcal I_k^c,\ \widetilde H_{\widehat{\mathcal H}_{-(N_0-9)}}<\widetilde H_{\widehat{\mathcal H}_{N_0(1+(1/11))}}].
\end{align}
In turn, to estimate the right hand side of (\ref{ineref}), we will introduce for reference purposes a one-dimensional coupling in the next remark.
\begin{remark}\label{remarkonedim}
For fixed $\omega \in\Omega$, we consider the one-dimensional random walk $(M_n)_{n\geq0}$ with absorbing barriers in $l_i-1:=-N_0-1$ and $l_j+1:=N_0(1+(1/11))+1$, and law $\widehat P_{m}$ where $m\in[l_i-1,l_j+1]$, such that
\begin{align*}
&\mbox{For $i\in [l_i,l_j]$, and $n\geq0$, we define transitions:}\\
&\widehat P_m[M_{n+1}=i+1|M_n=i]=1-\widehat P_m[M_{n+1}=i-1|M_n=i]:=\frac{1}{1+\rho_i}.\\
& \mbox{For $n\geq0$, the starting point is $m$ and the absorbing barriers are given by:}\\
&\widehat P_m[M_0=m]=1,\\
&\widehat P_m[M_{n+1}=l_i-1|M_n=l_i-1]=\widehat P_m[M_{n+1}=l_j+1|M_n=l_j+1]=1.
\end{align*}
This establishes a coupling between the \textit{actual random walk} $(X_n)_{n\geq0}$ and the one-dimesional $(M_n)_{n\geq0}$.

Roughly speaking, for fixed $y\in \widetilde B_{1,k+1}$ the one-dimensional random walk $(M_n)_{n\geq0}$ has the worst choice for the stationary transition $\widehat P_{l_k}[M_{n+1}=i+1|M_n=i]=:\alpha_i, \ i\in [-N_0,N_0(1+(1/11))]$ (cf. \ref{defrhoi})), when we consider the movement of $(X_n)_{n\geq0}$ along the event $\{\mathcal I_k^c,\ \widetilde H_{\widehat{\mathcal H}_{i}}<\widetilde H_{\widehat{ \mathcal H}_j}\}$, for $i<j$. It is now straightforward to see that for any point $x\in \widehat{\mathcal H}_m$, where $i\leq m\leq j$ we have
\begin{equation}\label{ineoneddim}
P_{x,\omega}[\mathcal I_k^c,\ \widetilde H_{\widehat{\mathcal H}_i}<\widetilde H_{\widehat{\mathcal H}_j}]\leq \widehat P_m[(M_n)_{n\geq0}\mbox{ hits }i\mbox{ before }j].
\end{equation}
The associated Poisson equation is:
$$
\begin{array}{c}
    \mathfrak Q_m:=P_{m}[(M_n)_{n\geq0}\mbox{ hits }i\mbox{ before }j]=\alpha_i \mathfrak Q_{m+1}+(1-\alpha_i)\mathfrak Q_{m-1}, \ m\in(i,j) \\
    \mathfrak Q_i=1,\mbox{  and  }  \mathfrak Q_j=0.
\end{array}
$$
The system above has unique solution (cf. \cite{Ch60} pp. 67-71):
\begin{equation}\label{qmoned}
\mathfrak Q_m=\frac{\sum_{m\leq n\leq j}\prod_{n<l\leq j}\rho_l^{-1}}{\sum_{i\leq n\leq j}\prod_{n<l\leq j}\rho_l^{-1}}.
\end{equation}
Therefore, in view of (\ref{ineoneddim}) we get
\begin{equation}\label{inerem1dfin}
\sup_{\substack{x\in \widehat{\mathcal H}_k}}P_{x,\omega}[\mathcal I_k^c,\ \widetilde H_{\widehat{\mathcal H}_i}<\widetilde H_{\widehat{\mathcal H}_j}]\leq \mathfrak Q_m,
\end{equation}
where $\mathfrak Q_m$ has the expression in display (\ref{qmoned}).
\end{remark}
We apply the estimate (\ref{inerem1dfin}) to inequality \ref{inecase2} to find that for any $y\in\widetilde B_{1,k+1}(0)$,
\begin{align}
\nonumber
P_{y,\omega}[\mathcal I_k^c,X_{T_{B_{2,k+1}(0)}}\cdot\ell\leq -L_{k+1}]&\leq \sup_{\substack{z_2\in\widehat{\mathcal H}_{-11}}}P_{z_2,\omega}[\mathcal I_k^c,\ \widetilde H_{\widehat{\mathcal H}_{-N_0}}<\widetilde H_{\widehat{\mathcal H}_{-10}}]\\
\nonumber
&\leq\frac{\sum_{-11\leq n\leq -10}\prod_{n<j\leq -10}\rho_{j}^{-1} }{\sum_{-N_0\leq n\leq -10}\prod_{n<j\leq -10}\rho_{j}^{-1}}\\
\label{inecase2fin}
&\leq \left(\frac{2}{\kappa}e^{-c_kL_k}\right)^{N_0-11}
\end{align}
provided that $L_0\geq\nu_2$ for certain constant $\nu_2>0$. We have used (\ref{inecase2first}) and the induction hypothesis (\ref{claimquengood}) to get (\ref{inecase2fin}).

\begin{enumerate}[(i)]
\setcounter{enumi}{2}
\item \label{case3quen} Case $\widetilde i\in(4,N_0-9)$.
\end{enumerate}
In this case, we have an in-between \textit{hole of three possible bad boxes}. For an arbitrary $y\in\widetilde B_{1,k+1}(0)$, we define the sets $\widehat{\mathcal H}_{i}$, where $i\in[-N_0,N_0(1+(1/11))]$, as in case (\ref{case1quen}). An analogous argument using the Markov property as the one given in cases (\ref{case2quen}) and (\ref{case1quen}), shows that for an arbitrary $y\in\widetilde B_{1,k+1}(0)$
\begin{align}
\label{inecase3first}
&P_{y,\omega}[\mathcal I_k^c, X_{T_{B_{2,k+1}(0)}}\cdot\ell\leq -L_{k+1}] \\
\nonumber
&\leq \sup_{\substack{z_1\in\widehat{\mathcal H}_0}}P_{z_1,\omega}[\mathcal I_k^c, \widetilde H_{\widehat{ \mathcal H}_{-\widetilde i}}<\widetilde H_{\widehat{ \mathcal H}_{N_0(1+(1/11))}}] \sup_{\substack{z_2\in \widehat{\mathcal H}_{-(\widetilde i+6)}}}P_{z_2,\omega}[\mathcal I_k^c, \widehat{\mathcal H}_{-N_0}<\widehat{\mathcal H}_{-(\widetilde i+5)}].
\end{align}
We apply Remark \ref{remarkonedim} on the first term to the right side of inequality (\ref{inecase3first}), and we get the estimate
\begin{align}
\label{inefthestcase3}
&\sup_{\substack{z_1\in\widehat{\mathcal H}_{0}}}P_{z_1,\omega}[\mathcal I_k^c, \widetilde H_{\widehat{ \mathcal H}_{-\widetilde i}}<\widetilde H_{\widehat{ \mathcal H}_{N_0(1+(1/11))}}]\\
\nonumber
&\leq \frac{\sum_{\substack{0\leq n\leq N_0(1+(1/11))}}\prod_{\substack{n<j\leq N_0(1+(1/11))}}\rho_{j}^{-1}}{\sum_{\substack{-\widetilde i\leq n\leq N_0(1+(1/11))}}\prod_{\substack{n<j\leq N_0(1+(1/11))}}\rho_{j}^{-1}}.
\end{align}
Furthermore, we use the inequality (\ref{inesuprho}) along with the induction assumption (\ref{claimquengood}) into inequality (\ref{inefthestcase3}) to find that
\begin{equation}\label{inecase3term1}
\sup_{\substack{z_1\in\widehat{\mathcal H}_{0}}}P_{z_1,\omega}[\mathcal I_k^c, \widetilde H_{\widehat{ \mathcal H}_{-\widetilde i}}<\widetilde H_{\widehat{ \mathcal H}_{N_0(1+(1/11))}}]\leq \left(\frac{2}{\kappa}e^{-c_kL_k}\right)^{\widetilde i-1},
\end{equation}
provided that $L_0\geq\nu_3$, where $\nu_3>0$ is certain positive constant.

\noindent
A quite similar argument as the given above, with the help of Remark \ref{remarkonedim}, the induction hypothesis (\ref{claimquengood}) and the inequality (\ref{inesuprho}) provides the estimate,
\begin{align}
\nonumber
&\sup_{\substack{z_2\in \widehat{\mathcal H}_{-(\widetilde i+6)}}}P_{z_2,\omega}[\mathcal I_k^c, \widehat{\mathcal H}_{-N_0}<\widehat{\mathcal H}_{-(\widetilde i+5)}]\\
\label{inecase3term2}
&\leq \left(\frac{2}{\kappa}e^{-c_kL_k}\right)^{N_0-\widetilde i-7}
\end{align}
provided that $L_0\geq\nu_4$, where $\nu_4>0$ is certain positive constant.

Thus, combining both upper bounds (\ref{inecase3term1})-(\ref{inecase3term2}), in virtue of the inequality (\ref{inecase3first}), for any point $y\in\widetilde B_{1,k+1}(0)$ we obtain
\begin{equation}\label{inecase3fin}
P_{y,\omega}[\mathcal I_k^c, X_{T_{B_{2,k+1}(0)}}\cdot\ell\leq -L_{k+1}]\leq  \left(\frac{2}{\kappa}e^{-c_kL_k}\right)^{N_0-8}
\end{equation}
This finishes the analysis of case (\ref{case3quen}) and close our required estimates for the probability in (\ref{probleft}).

\smallskip
We now combine the estimates given in cases (\ref{case1quen})-(\ref{case3quen}) along with the \textit{lateral} estimate (\ref{inelatfin}). Specifically, in view of inequality (\ref{inelatfin}), we use the inequalities displayed in (\ref{inefinfincase1})- (\ref{inecase2fin})- (\ref{inecase3fin}), in order to see that
$$
\sup_{\substack{y\in \widetilde B_{1,k+1}(0)}}P_{y,\omega}[X_{T_{B_{2,k+1}}}\notin \partial^+B_{2,k+1}(0)]\leq 2\left(\frac{2}{\kappa}e^{-c_kL_k}\right)^{N_0-9}\leq e^{-\frac{c_kL_{k+1}}{4}}
$$
provided that $L_0>\nu_1$, for certain constant of the model $\nu_5>0$. We have used our scaling choice (\ref{scale1})-(\ref{scale3}), which implies in particular that $N_0-9>N_0/2$. Furthermore, we have chosen $L_0$ large enough so that
$$
2\left(\frac{2}{\kappa}\phi^{\frac{1}{2}}(L_0)\right)^{\frac{N_0}{2}}\leq e^{-c_1L_1}=\phi^{\frac{N_0}{8}}(L_0).
$$
This ends the induction and proves (\ref{claimquengood}) by using the expression of constant $(c_k)_{k\geq0}$ in (\ref{defckquen}).
\end{proof}
We now proceed to combine Proposition \ref{propannebad} and Proposition \ref{propquengood} to localize \textit{a generic box of scale} $L$, for a large number $L$ between two consecutive boxes of scales $L_k$ and $L_{k+1}$. We start with introducing an auxiliary stretched exponential condition.
\begin{definition}\label{defTGam}
Let $\gamma\in(0,1]$, $\ell\in\mathbb S^{d-1}$ and $R$ be a rotation of $\mathbb R^d$, such that $R(e_1)=\ell$.
For $L>0$ we introduce box $B_{0,L}$ by
$$
B_{0,L}=R\left((-L,L)\times(-2L^3,2L^3)^{d-1}\right)\cap \mathbb Z^d.
$$
\noindent
We say that condition $(\mathfrak T^\gamma)|\ell$ holds, if
\begin{equation}\label{conTga}
\limsup_{\substack{L\rightarrow\infty}}L^{-\gamma}\ln \left(P_0[X_{T_{B_{0,L}}}\notin \partial^+B_{ 0,L}]\right)<0.
\end{equation}
\end{definition}
Let us mention that condition $(\mathfrak T^\gamma)|\ell$ is a priori weaker than condition $(T^\gamma)|\ell$ in Definition \ref{defTs}. The detail can be found in Lemma 2.2 of \cite{Gue19} for the case $\gamma=1$ and Appendix of \cite{GVV19} for $\gamma\in(0,1)$.

We let constant $\lambda_1$ in Defintion \ref{defweakcond} as follows
$$
\lambda_1:=\min\{h\}
$$
Roughly speaking, we ask the minor requirement in order to satisfy Propositions \ref{propannebad} and \ref{propquengood}.
\begin{theorem}\label{mainth3}
Assume that condition $(\mathcal W_{c,M})|\ell$ holds. Then there exists a constant $\gamma>0$, such that condition $(\mathfrak T^\gamma)|\ell$ holds.
\end{theorem}
\begin{proof}
Since $(\mathfrak W_{c,M})|\ell$ holds for $M>1/\lambda_1$, we consider scales (\ref{scale1})-(\ref{scale3}) with $L_0=M$ and the renormalization construction provided by the successive blocks in $\mathfrak B_k$ with centres at points in the set $\mathfrak L_k$, with $k\geq0$.
We let 
$$
\gamma:=\ln(2)/(2\ln(N_0))\in (0,1)
$$ 
and consider for large $L$  the first integer $k>0$ such that $L_k\leq L$. We introduce the environment event $\mathfrak G_k$ of good boxes of scale $k$ intersecting $B_{0,L}$, defined by
\begin{align}
\label{goodenv}
\mathfrak G_k:= \big\{ & \forall B_{2,k}(w), \ w\in \mathfrak L_k,\\
\nonumber
&\dot B_{1,k}(w)\subset B_{2,k+1}(0)\Rightarrow B_{2,k}(w) \mbox{ is } L_k-\mbox{ \textit{Good}} \big\}
\end{align}
We then split the required expectation into two terms,
\begin{align}
\label{inesplitTG}
&P_0[X_{T_{B_{0,L}}}\notin \partial^+B_{0,L}]
\leq \mathbb E[\mathds 1_{\mathfrak G_k^c}]+\mathbb E[P_{0,\omega}[X_{T_{B_{0,L}}}\notin \partial^+B_{0,L}]\mathds 1_{\mathfrak G_k}].
\end{align}
Observe that using the Proposition \ref{propannebad}, the first expectation on the right hand side of (\ref{inesplitTG}) after a rough counting argument, can be bounded from above by
\begin{equation}\label{ine1fORTG}
\mathbb E[\mathds 1_{\mathfrak G_k^c}]\stackrel{\mbox{Remark \ref{remarkcover}}}\leq (N_0(2+(1/11))+2)(5\widetilde c\widetilde N_0+2)^{d-1}e^{-\eta_1 2^k}.
\end{equation}
On the other hand, we introduce a strategy encoded by the stopping times $(H^i)_{i\geq0}$ and the random position $(Z_i)_{i\geq0}$ together with $(Y_i)_{i\geq0}$ defined by
\begin{align}
\nonumber
&H^0=0,\  Z_0=X_0,\ Y_0=\textit{ an arbitary poin in }\{z\in\mathfrak L_k:\ Z_0\in \widetilde B_{1,k}(z) \},\\
\nonumber
&H^1=T_{B_{2,k+1}(0)}\wedge T_{B_{2,k}(Y_0)},\ Z_1=X_{H^1},\ Y_1=\textit{ an arbitary poin in }\\
\nonumber
&\{z\in\mathfrak L_k:\ Z_1\in \widetilde B_{1,k}(z) \}. \\
\nonumber
&\mbox{Moreover, we recursively define for integer $i>1$},\\
\nonumber
& H^i=H^{i-1}+H^1\circ \theta_{H^{i-1}},\  Z_i=X_{H^i},\ Y_i=\textit{ an arbitary poin in }\\
\label{stopstralat}
&\{z\in\mathfrak L_k:\ Z_i\in \widetilde B_{1,k}(z) \}.
\end{align}
We also introduce the $(\mathcal F_n)_{n\geq0}-$ stopping time $S$ defined by
$$
S=\inf\left\{n\geq 0:\ X_n \in \partial B_{2,k}(Y_0)\setminus \partial^+ B_{2,k}(Y_0)\right\}.
$$
Notice that the following claim
\begin{align}
\label{ine2fORTG}
&\mathbb E\left[P_{0,\omega}[X_{T_{B_{0,L}}}\notin\partial^+ B_{0,L}]\mathds 1_{\mathfrak G_k}\right] \\
\nonumber
&\leq 1-\mathbb E\left[P_{0,\omega}\left[\bigcap_{\substack{0\leq i< N_0}}\theta_i^{-1}\{H^1<S\}\right]\mathds 1_{\mathfrak G_k}\right],
\end{align}
holds. Indeed for large $L$ one has that $3\widetilde c \widetilde L_k(N_0-1)+4\widetilde c \widetilde L_k<2L^3$ (cf. (\ref{scale1})-(\ref{scale3})). Therefore, we have that $\mathbb P-$ a.s.
$$
P_{0,\omega}[X_{T_{B_{0,L}}}\in\partial^+ B_{0,L}]\geq P_{0,\omega}\left[\bigcap_{\substack{0\leq i< N_0}}\theta_i^{-1}\{H^1<S\}\right].
$$
As a result of Proposition \ref{propquengood} on inequality (\ref{ine2fORTG}) we see that
\begin{align}
\label{ine2forTG2}
&\mathbb E\left[P_{0,\omega}[X_{T_{B_{0,L}}}\notin\partial^+ B_{0,L}]\mathds 1_{\mathfrak G_k}\right]\\
\nonumber
&\leq 1-\left(1-e^{-\eta_2v^k}\right)^{N_0}\leq N_0e^{-\eta_2v^k}.
\end{align}
In view of applying (\ref{ine1fORTG}) and (\ref{ine2forTG2}) into (\ref{inesplitTG}), we find that
\begin{align*}
&P_0[X_{T_{B_{0,L}}}\notin \partial^+B_{0,L}]\leq 2(N_0(2+(1/11))+2)(5\widetilde c\widetilde N_0+2)^{d-1}e^{-\eta_1 2^k}\\
&\leq 2(N_0(2+(1/11))+2)(5\widetilde c\widetilde N_0+2)^{d-1}\exp\left(-\eta_1\left(\frac{L}{L_0}\right)^{\frac{\ln(2)}{2\ln(N_0)}}\right)=e^{-\eta_3L^{\gamma}},\\
\end{align*}
for certain constant $\eta_3:=\eta_3(L_0,d)>0$. The last inequality proves the claim in the theorem.
\end{proof}
\begin{proof}[Proof of Theorem \ref{mainth1}]
The proof of second part in Theorem \ref{mainth1} is concerned with a straightforward geometric argument and will be omitted.
We now conclude the proof of Theorem \ref{mainth1}. Observe that $(\mathfrak T^\gamma)|\ell$ plainly implies condition $(T^{\Gamma(N)})|\ell$ of \cite{GVV19}. Therefore Theorem 5.11 in \cite{GVV19} and the present Theorem \ref{mainth3} prove the equivalence between conditions $(\mathcal W_{c,M})|\ell$ and $(T')|\ell$. In the i.i.d. random environment case, we further apply the main result of \cite{GR18} to finish the proof of Theorem \ref{mainth1} for dimension $d\geq2$. The one dimensional case is explained in the next section.
\end{proof}

\medskip
\section{One dimensional finite argument: Proof of Corollary \ref{coronedim}}\label{seconedim}
In this section we will prove Corollary \ref{coronedim}. This result is well-known from the solution of the Poisson's equation as in Remark \ref{remarkonedim} or the one dimensional effective criterion. Nevertheless we display a new argument to show a possible new path which might be used to prove Conjecture \ref{conjec2} in higher dimensional case.
\begin{proof}[First Proof of Corollary \ref{coronedim}.]
In virtue of Proposition 2.6 in \cite{Sz01} the equivalence between transience along direction $e_1$ and condition $(T)|e_1$ was proven.  It is a simple matter to show
that arbitrary decay implies condition $(\mathcal W)|e_1$, since the boundary of the box $\widetilde B_1(c,M)$ are two points.
Therefore we only need to prove that $(\mathcal W)_{c,M}|e_1$ implies condition $(T)|e_1$.  To this end, we observe that Theorem \ref{mainth1} proves that $(\mathcal W)_{c,M}|e_1$ implies $(T')|e_1$. However $(T')|e_1$ implies transience along $e_1$ in any dimension, and as a result of Proposition 2.6 we have condition $(T)|e_1$.
\end{proof}
\begin{proof}[Second Proof of Corollary \ref{coronedim}]
Let $d$ be the dimension which will be essentially equals $1$.  We first assume that the random environment at site $0$ (and therefore at any site) takes finite many values $\omega_i(0, \cdot)\in\mathcal P_{\kappa}, \ i\in[1,m]$ on the simplex (\ref{simplex}) with probabilities $p_i\in (0,1),\ i\in[1,m]$, respectively. We consider the probability $\mu_m$ on $\mathcal P_\kappa$,
$$
\mu_m:=\sum_{i=1}^m p_i\mathds 1_{\{\omega(0, \cdot)=\omega_i(0, \cdot)\}}
$$
and the corresponding product measure $\mathbb P^m:=\mu_m^{\mathbb Z^d}$ on $\Omega$. Assume for the time being, the following claim:

\textit{For every finite environment as above, satisfying condition $(\mathcal W)_{c,M}|e_1$  there exists a constant $c>0$ not depending on $m$ such that for all large $L$ one has
\begin{equation}\label{claimoned}
P_0^m[X_{T_{U_L}}\notin \partial^+U_L]\leq e^{-cL},
\end{equation}
where $U_L$ and $\partial^+U_L$ are defined in the statement of Corollary \ref{coronedim} and $P_0^m:=\mathbb P^m\otimes P_{0,\omega}$.}
As the proof will show, the ballistic hypothesis above can actually be relaxed to the existence of $L_0>0$ satisfying
$$
P_0^m[X_{T_{U_{L_0}}}\notin \partial^+U_{L_0}]<1.
$$
We turn now to prove the Corollary starting from the claim. Let $L_0>0$ be a fixed number large enough such that for the original one-dimensional environment $\omega$ we have
\begin{equation}\label{probtranone}
P_0[X_{T_{U_{L_0}}}\notin \partial^+U_{L_0}]< \frac{\lambda_1}{2}.
\end{equation}
Notice that by definition $P_0[X_{T_{U_L}}\notin \partial^+U_L]=\mathbb E[P_{0,\omega}[X_{T_{U_L}}\notin \partial^+U_L]]$. Therefore, since the random variable inside is a function of the transitions in the slab, the big picture is to approximate each environment at each site for finite many sites depending on fixed $L_0$ and then make the approximation finer.

\noindent
For $n\in \mathbb N$, we consider \textit{finite approximations} $\omega_m(0,\cdot)$ \textit{of size} $m=n^d$ at size $0$ as follows (recall (\ref{simplex})):
\begin{align}
\nonumber
&\omega(\pm e_i,k):=k\frac{\kappa}{m}\mathds 1_{\{\omega(0,\pm e_i) \in [k\frac{\kappa}{m+1},(k+1)\frac{1-d\kappa}{m+1}]\}}, \mbox{   for $i\in[1,d]$, $k\in[1,m]$}\\
\nonumber
&c_{k_1,k_2,\ldots, k_{2d},m}^{-1}:=\sqrt{\sum_{k\in \{k_1,k_2,\ldots ,k_{2d} \}}\frac{k\kappa}{m}}\mbox{   for $k_i\in[1,m]$, $i\in[1,2d]$}\\
\nonumber
&\omega_m (0,\cdot):=\sum_{k_1,k_2,\ldots ,k_{2d}\in [1,m]}c_{k_1,\ldots,k_{2d},m}\omega(e_1,k_1)\omega(e_2,k_2)....\omega(-e_d, k_{2d}).
\end{align}
We denote by $\omega_m$ the environment generated in the product space by this \textit{environment at site} $0$. 
The random variable $P_{0,\omega}[X_{T_{U_L}}\notin \partial^+U_L]$ is certain function of the environment $\omega$ when restricted to $U_{L}$. Thus, since $L_0$ is fixed, we have that $\mathbb P-$ a.s.
$$
P_{0,\omega_m}[X_{T_{U_{L_0}}}\notin \partial^+U_{L_0}]\rightarrow P_{0,\omega}[X_{T_{U_{L_0}}}\notin \partial^+U_{L_0}]
$$
as $m\rightarrow\infty$. Above, the left hand side denote the random environmental function evaluated at the finite valued environment $\omega_m$. Then, we pick an $m$ large enough so that:
$$
\left|\mathbb E\left[P_{0,\omega_m}[X_{T_{U_{L_0}}}\notin \partial^+U_{L_0}]- P_{0,\omega}[X_{T_{U_{L_0}}}\notin \partial^+U_{L_0}]\right]\right|<\frac{\lambda_1}{2}.
$$
As a result, in virtue of (\ref{probtranone})we have that 
$$
\mathbb E[P_{0,\omega_m}[X_{T_{U_{L_0}}}\notin \partial^+U_{L_0}]]<\lambda_1.
$$
In turn, as a further result of applying (\ref{claimoned}) we get the result of Corollary \ref{coronedim}. 

\noindent
We now turn to prove inequality (\ref{claimoned}). We first observe that for $L=cL_0$ where $L_0>3\sqrt{d}$ and $c>1$ a large integer, on the event $\{X_{T_{U_{L}}}\notin \partial^+U_{L} \}$ the $(\mathcal F_n)_{n\geq0}-$ stopping time $\widetilde T_{-L}^{e_1}$ is $P_0^m-$ a.s finite, where $P_0^m:=\mathbb P^m\otimes P_{0,\omega}$ and the random time $\widetilde T_{-L}^{e_1}\circ \theta_{S_L^{e_1}:=\sup\{n\geq0: \ (X_n-X_0)\cdot e_1 \geq 0\}}+\widetilde T_{-L}^{e_1}\wedge \widetilde T_{L}^{e_1}$ is also $P_0^m-$ a.s finite.

\noindent
Therefore, we define the for $a\in\mathbb R$ the strip 
$$
\mathcal H_{a}:=\{z\in\mathbb Z^d:\  \exists z' \in \mathbb Z^d \ |z-z'|_1=1 (z\cdot e_1-aL_0)(z'\cdot e_1-aL_0)\leq 0  \},
$$
and a standard Markov chain estimate gives
\begin{align}
\nonumber
P_{0,\omega}[X_{T_{U_{L}}}\notin \partial^+U_{L}]\leq & \sum_{x\in \mathcal H_{-(c-1)}}P_{0,\omega}[\widetilde T_{-(c-1)L_0}^{e_1}< T_{L}^{e_1}, X_{\widetilde T_{-(c-1)L_0}^{e_1}}=x]\\
\label{firestonedquen}
&\times P_{x,\omega}[\widetilde T_{-L}^{e_1}<T_L^{e_1}, \  S \wedge \widetilde T_L^{e_1}<\infty].
\end{align}
In turn, for $x\in \mathcal H_{-(c-1)}$ we see that
\begin{align}\label{decsumoned}
P_{x,\omega}[\widetilde T_{L}^{e_1}<T_L^{e_1}, \  S \wedge \widetilde T_L^{e_1}<\infty]=\sum_{n\geq0}P_{x,\omega}[\widetilde T_{L}^{e_1}<T_L^{e_1}, \  S \wedge \widetilde T_L^{e_1}=n]
\end{align}
We introduce the cube $C(x,n)$ of size $n\in\mathbb N$ centered at $z\in \mathcal H_{-(c-1)}$, and its \textit{central boundary} $\partial^0 C(z,n)$ via
$$
C(z,n):=z+[-n,n]^d \mbox{   and   } \partial^0 C(z,n):=\mathcal H_{-(c-1)}\cap C(z,n).
$$
We now decompose according to the $m$ values of $\omega$ at each site of $C(x,n)$ the probability inside the sum in (\ref{decsumoned}). We denote the set of environmental configuration of $C(z,n)$ by $\mathcal W_{z,n}$, with the hopeful clear notation
\begin{align*}
\mathcal W_{z,n}:=\{&w=(\omega_{i_1}(x_1,\cdot),\omega_{i_2}(x_2,\cdot),\ldots, \omega_{i_{|C(z,n)|}}(x_{i_{|C(z,n)|}},\cdot) ),\\
& i_j\in[1,m], \ \cup_{j} x_{i_j}=C(z,n)\}.
\end{align*}
For $x\in \mathcal H_{-(c-1)}$ and integer $n\geq0$ we find that, 
\begin{align}
\nonumber
&P_{x,\omega}[\widetilde T_{-L}^{e_1}<T_L^{e_1}, \  S \wedge \widetilde T_{-L}^{e_1}=n]\\
\nonumber
&=\sum_{w\in\mathcal W_{x,n}}P_{x,\omega}[\widetilde T_{-L}^{e_1}<T_L^{e_1}, \  S \wedge \widetilde T_L^{e_1}=n, (\omega_x)_{x\in {C(x,n)}}= w]\\
\nonumber
&\leq\sum_{w\in\mathcal W_{x,n}, y \in \partial^0 C(x,n)}P_{x,\omega}[(\omega_x)_{x\in {C(x,n)}}= w,\ X_{S \wedge \widetilde T_{-L}^{e_1}}=y, S \wedge \widetilde T_L^{e_1}=n]\\
\label{inonedexp}
&\times P_{y,\omega}[(\omega_x)_{x\in {C(x,n)}}= w,\ \widetilde T_{-L}^{e_1}<\widetilde T_{-(c-1)L}^{e_1} ]
\end{align}
The crucial point is that calling $P_{s,x}$ the simple-symmetric random walk law starting from $x\in \mathbb Z^d$ we have that
\begin{align}\label{ineq1oned}
&sup_{w\in\mathcal W_{x,n}}P_{x,\omega}[(\omega_x)_{x\in {C(x,n)}}= w,\ X_{S \wedge \widetilde T_{-L}^{e_1}}=y, S \wedge \widetilde T_L^{e_1}=n]\\
\nonumber
&=P_{s,x}[X_{S \wedge \widetilde T_{-L}^{e_1}}=y, S \wedge \widetilde T_L^{e_1}=n],
\end{align}
along with for $y \in \mathbb Z^d$, denoting by $\mathbb E^m$ the expectation with respect to $\mathbb P^m$, we have 
\begin{align}\label{ineq2oned}
\mathbb E^m [\sum_{w\in\mathcal W_{x,n}} P_{y,\omega}[(\omega_x)_{x\in {C(x,n)}}= w,\ \widetilde T_{-L}^{e_1}<\widetilde T_{-(c-1)L}^{e_1}]]\leq \lambda_1 .
\end{align}
We also notice that for $x\in \mathcal H_{-(c-1)L_0}$ the function: 
$$
P_{0,\omega}[\widetilde T_{-(c-1)L_0}^{e_1}< T_{L}^{e_1}, X_{\widetilde T_{-(c-1)L_0}^{e_1}}=x]
$$
and for $y\in \partial^0 C(x,n)$: 
$$
\sum_{w\in\mathcal W_{x,n}} P_{y,\omega}[(\omega_x)_{x\in {C(x,n)}}= w,\ \widetilde T_{-L}^{e_1}<\widetilde T_{-(c-1)L}^{e_1}],
$$
are independent random variables under $\mathbb P^m$.

\noindent
Using the last remark, together with (\ref{ineq1oned}) and (\ref{ineq2oned}) in ((\ref{inonedexp}) ) and going back to (\ref{firestonedquen}) we obtain
\begin{align*}
&\mathbb E^m[P_{0,\omega}[X_{T_{U_{L}}}\notin \partial^+U_{L}]]\\
&=\mathbb E^m[P_{0,\omega}[\widetilde T_{-L}^{e_1}<T_{L}^{e_1}]]\leq \lambda_1 \mathbb E^m[P_{0,\omega}[\widetilde T_{-(c-1)L_0}^{e_1}<T_{L}^{e_1}]]
\end{align*}
and by a standard induction argument we get
$$
\mathbb E^m[P_{0,\omega}[X_{T_{U_{L}}}\notin \partial^+U_{L}]]\leq \lambda_1^{c}=e^{-\frac{\ln(1/\lambda)}{L_0}L}
$$ 
which ends the proof of claim (\ref{claimoned}).
\end{proof}
It is an open question to provide a replication of this argument in the multidimensional case and/or find a weaker ballisticity condition than the given in the 
present article. Indeed, it remains as a challenging question whether a class of multidimensional analogous bound to (\ref{ineq1oned}) might be proven.

\end{document}